\documentclass[12pt,reqno]{amsart}
\usepackage{amsmath}
\usepackage{amssymb,latexsym}
\usepackage[matrix,arrow]{xy}
\usepackage[usenames]{color}
\usepackage{slashed}
\usepackage{parskip}
\def\rjk{R_{j,k}}

\def\rpq{R_{p,q}^k}
\def\rpl{R_{p,l}^j}
\def\gnk{G_{n,k}}
\def\agnk{ G(n,k)}
\def\agnj{ G(n,j)}
\def\gnj{G_{n,j}}

\def\rk{\bbr^k}

\def\rnk{\bbr^{n-k}}

\def\Cal{\mathcal}

\def\gnk{G_{n,k}}

\def\bbr{{\Bbb R}}

\def\bbd{{\Bbb D}}

\def\Pr{{\hbox{\rm Pr}}}

\def\gnk{G_{n,k}}

\def\rn{\bbr^n}

\def\rk{\bbr^k}
\def\rnk{\bbr^{n-k}}

\def\part{\partial}
\def\intl{\int\limits}
\def\b{\beta}

\def\Gam{\Gamma}

\def\a{\alpha}
\def\om{\omega}

\def\Del{\Delta}

\def\vp{\varphi}

\def\g{\gamma}
\def\gam{\gamma}

\def\sig{\sigma}
\def\lam{\lambda}

\def\z{\zeta}
\def\e{\varepsilon}

\def\t{\tau}

\def\nn{\nonumber}

\newtheorem{theorem}{Theorem}[section]
\newtheorem{lemma}[theorem]{Lemma}

\theoremstyle{definition}

\newtheorem{example}[theorem]{Example}

\theoremstyle{remark}
\newtheorem{remark}[theorem]{Remark}

\theoremstyle{corollary}
\newtheorem{corollary}[theorem]{Corollary}

\numberwithin{equation}{section}

\newcommand{\be}{\begin{equation}}
\newcommand{\ee}{\end{equation}}

\newcommand{\bea}{\begin{eqnarray}}
\newcommand{\eea}{\end{eqnarray}}
\newcommand{\Bea}{\begin{eqnarray*}}
\newcommand{\Eea}{\end{eqnarray*}}

\begin{document}
\title{Strichartz transforms with  Riesz potentials and Semyanistyi integrals}
\author []{ Yingzhan Wang}
\email{wyzde@gzhu.edu.cn}
\address{School of Mathematics and Information Science, Guangzhou University, Guangzhou, China}

\thanks{This work  was supported by the Project of  Guangzhou  Science and Technology Bureau (No. 202102010402).}.

\subjclass[2000]{ 44A12, 47G10}

\keywords{Radon transforms,Grassmann manifolds,Riesz potentials, Semyanistyi integrals}

\begin{abstract}In this paper, we study  the general orthogonal Radon transform $R_{p,q}^k$ first studied by R.S Strichartz in \cite{Stri}. An
   sharp    existence condition of $R_{p,q}^k f$ on $L^p$-spaces will be given. Then we devote to   the relation formulas connecting  Strichartz transform     $R_{p,q}^k$ and   Semyanistyi integrals. We   prove the corresponding  Fuglede type formulas, through which
   a number of  explicit inversion formulas for $R_{p,q}^k f$ will be given. Different from the ``inclusion" Radon transform and ``Gonzalez" type orthogonal  transform, Strichartz transform is more complicated. Our conclusions generalize the corresponding results of the  two particular cases above.
\end{abstract}
\maketitle
\renewcommand{\sectionmark}[1]{}

\section{Introduction}Problems of integral geometry related to Grassmann manifolds constitute a core of different branches of mathematics and arise in many applications, which can be found in the book by Helgason \cite{He}, Gonzalez\cite{G3}
and Rubin\cite{Ru3}, containing many references on this subject.
Here we concern the orthogonal Radon transform on affine Grassmannians  first studied by R.S Strichartz in \cite{Stri}. A background information about this kind of transforms can be found in  \cite{G1,G2,G4,GK,OR,Ru1,Ru2,RW1,RW2,Sem} etc..  Let $\agnj$ and $\agnk$ be a pair of Grassmannian bundles of $j$-dimensional and $k$-dimensional affine planes
in $\rn$.
For $p\ge 0$, $q\ge 0$, $l\ge 0$, that satisfying  $p+q=j$ and $p+l=k$,
we call $j$-plane $\t\in \agnj$ and $k$-plane $\z\in \agnk$  incident if $\t$  intersect
$\z$ orthogonally in a $p$-plane.
Let $
\hat{\z}$ be the set of all $j$-planes incident to $\z$ and $\hat{\t}$ the set of $k$-planes incident to $\t$. Then for
 a good function $f$ on $\agnj$, the orthogonal Radon transform $\rpq f$ is a function  on $\agnk$, where the value $\rpq f(\z)$, $\z\in \agnk$ is defined as an integral of $f$ over the set of all $j$-planes incident to $\z$. Similarly, for a good function $g$ on $\agnk$, the corresponding Radon transform $\rpl g(\t)$ is a function on $\agnj$ that integrates $g$ over the set of all $k$-planes incident to $\t$.
  Formly,
\bea\label{defor}
\rpq f(\z)=\intl_{\hat{\z}}f(\t)d\t\,, \quad \rpl g(\t)=\intl_{\hat{\t}}f(\z)d\z\,,
\eea
where $d\t$  and $d\z$ are the invariant measure, see (\ref{rpq}), (\ref{rpl}) for precise meaning.

There
have been numerous   publications  devoted to its two particular  cases. When $q=0$, our transform reduces to the case of  ``inclusion" transform $\rjk$  studied by Gonzalez and Kakehi \cite{GK}, which takes  functions on the Grassmannian of $j$-dimensional affine planes in $\rn$ to functions on a similar manifold of $k$-dimensional planes by integration over the set of all $j$-planes that contained in a given $k$-plane. Under the condition $k-j$ is even, they  studied the range characterization  and inversion problems, using the  Lie algebra language  on only smooth rapidly decreasing
functions.   A sharp existence condition for this  transform was given by B. Rubin \cite{Ru2}, where Rubin inverted these transforms  in the framework of Lebesgues
spaces for arbitrary $k-j>0$ provided that these operators are injective. But the method  of  stereographic projection makes all formulas more complicated because of inevitable
weight factors. The inversion formulas for continuous functions and general $L^p$
functions were obtained by us directly through part Radon transform \cite{RW1}, which is a quite different approach.

 The case of  $p=0$  reduces to the $orthogonal$ transform $R_j^k$ studied by Gonzalez \cite{G1,G2}, which integrates a function on $\agnj$ over all the $j$-planes orthogonal to a $k$-plane with only one point  intersection. (In the following statements, we call it ``Gonzalez" transform.) These publications deal with the inversion formulas of the  transforms corresponding to smooth functions, where Gonzalez also studied  the intertwining relations connecting  $R_j^k$ and invariant differential operators. The general intertwining relationship between these transforms and fractional integral operators were studied by the author and Rubin (see \cite{RW2,RW3,RW4}). Our approach  extends to the case of  ``inclusion" Radon transform talked above.  In the present paper, we will investigate the general Strichartz type Radon transforms in which the two corresponding orthogonal planes have a $p$-plane intersection,  under the
assumption that
\be p>0,\quad q>0,\quad l>0,\quad p+q+l<n\,.\nonumber\ee
For this general case, Strichartz \cite{Stri} mainly focused on  its harmonic properties. In \cite{G4}, Gonzalez studied the intertwining relations between Strichartz transforms and some invariant differential operators, where only smooth functions with compact support are considered. Here, our main concerns  are  (a) sharp conditions for
$f$ under which $\rpq f$ exists in the Lebesgue sense, (b) functional relations
connecting $\rpq$ with Riesz potentials,   Radon-John $k$-plane transforms,  Semyanistyi integrals and (c)
explicit inversion formulas for arbitrary dimensions $j$ and $k$.

Our paper is organized as follows. Section 2 contains some basic notations about fractional integral operators and Radon-John transforms, ``inclusion" Radon transforms, ``Gonzalez" transforms and Strichartz transforms, etc.. Some known necessary   intertwining  relations  connecting   them will be introduced.  In section 3, we mainly study   Strichartz transforms on radial functions, through which a sharp existence condition for this transform will be given in section 4. In the last section, we devote to the  intertwining relations connecting  Strichartz transforms, Semyanistyi integrals and potential integral  operators. We  prove the Fuglede type theorems, through which some inversion formulas can be obtained  when function $f$  belongs to the range of the $j$-plane transform.

\section{Preliminaries}

\subsection{Riesz Potentials and Erd\'{e}lyi--Kober
Fractional Integrals}
In this part, we introduce the basic knowledge about the  Riesz potential operators first, which can be regard as a negative power of the minus-Laplacian
\be\label{sdyt1} (I_n^\a) \sim (-\Del_n)^{-\a/2}\,,\quad \Del_n=\frac{\part^2}{\part x_1^2}+\frac{\part^2}{\part x_2^2}+...+\frac{\part^2}{\part x_n^2}\,.\ee
Explicitly, for a good function $f$ on $\rn$, the Riesz potential of a positive $\a$ is defined by
\be\label{rpot} (I_n^\a f)(x)=\frac{1}{\gamma_n(\a)}
\intl_{\bbr^n}
 \frac{f(y)\,dy}{|x-y|^{n-\a}},\qquad
  \gamma_n(\a)=
  \frac{2^\a\pi^{n/2}\Gamma(\a/2)}{\Gamma((n-\a)/2)};
  \ee
\[ \a \neq n, n+2, n+4, \ldots \,.\]
If $f\in L^p (\rn)$, $1\le p <n/\a$, then $(I_n^\a f)(x)<\infty$ for almost all $x$, and the bounds for $p$ are sharp; see \cite[Chapter III]{Ru3}.  For the left inverse of $I_n^\a$, denoted by $\bbd_n^\a$, numerous investigations are devoted to this
question; see [Rul, SKM] and references therein.
 \begin{theorem}\cite[Theorem 3.41]{Ru3}Suppose $\vp=I_n^\a f$, where $f\in L^p (\rn)$, $1\le p <n/\a$.  If $\mu$ is a radial finite complex Borel measure on $\rn$, satisfying
 \bea
 &&\intl_{|x|>1}|x|^\b d|\mu|(x)<\infty,\quad \text{for some } \b>\a\,;\\
 &&\intl_{\rn}|x|^j d\mu(x)=0,\quad \text{for } |j|=0,2,4,...2[\a/2];\\
 &&d_{\mu}(a)=\frac{\pi^{n/2} 2^{-\a}\varkappa(\a)}{\sig_{n-1}\Gam((n+\a)/2)}\neq 0
 \eea
 where \bea
 \varkappa(\a)&=&\intl_0^\infty t^{-\a/2-1}dt\intl_{\rn}e^{-t|y|^2}d\mu(y) \nonumber\\
{} \nonumber\\
&=& \left\{ \!
 \begin{array} {ll} \! \Gamma (-\a/2) \intl_{\rn}|y|^\a\,d\mu(y) \!  & \mbox{if $ \a \neq 2, 4, 6, \ldots $,}\\
{}\\
\displaystyle{\frac{2 (-1)^{\a/2+1}}{(\a/2)!}\intl_{\rn}|y|^\a\log|y|\,d\mu(y) \!} & \mbox{if $ \a = 2, 4, 6, \ldots \, ,$}
\end{array}
\right. \nonumber\eea
 then
\be\label{invs1a}(\bbd_n^\a
\vp)(x)\!\equiv\!\frac{1}{d_{\mu}(\a)}\intl_{0}^\infty\!
\frac{(\vp*\mu_t)(x)}{t^{1+ \a}}\, dt\!=\!
\lim\limits_{\e \to
0}\frac{1}{d_{\mu}(\a)}\intl_{\e}^\infty\!
\frac{(\vp*\mu_t)(x)}{t^{1+ \a}}\, dt \,,\ee
represents the left inverse of Riesz potential $I_n^\a f$, that is, $\bbd_n^\a\,I_n^\a\,f\,=\,f\,.$ The limit in (\ref{invs1a}) exists in the $L^p$ norm and in the almost everywhere sense. If $f\in C_0(\rn)\cap L^p(\rn)$, this limit is uniform on $\rn$.
 \end{theorem}

The inversion formula (\ref{invs1a}) is non-local. If $\a$ is an even integer, then the local inversion formula
 $(-\Del_n)^{\a/2}I_n^\a f =f$ is available under additional smoothness assumptions for $f$.

We also work with Erd\'{e}lyi-Kober type
fractional integrals of  positive order $\a$, which arise in  integral geometry and many other applications of Fractional Calculus. Detailed information can be found in
\cite[Section 2]{Ru3}. The well known two Erd\'{e}lyi-Kober type
fractional integrals have the following expressions,
\bea
\label{as34b12}%
(I^{\a}_{+, 2} f)(t)
&=&\frac{2}{\Gam
(\a)}\intl_{0}^{t} (t^{2} -r^{2})^{\a-1}f (r) \, r\, dr,\\
\label{eci}
(I^{\a}_{-, 2} f)(t)
&=&\frac{2}{\Gam
(\a)}\intl_{t}^{\infty}(r^{2} - t^{2})^{\a-1}f (r) \, r\,
dr.\quad
\eea
For these operators, we have the following existence theorems.
\begin{lemma}\cite[Lemma 2.42]{Ru3}
\label{lifa2}\

\textup{(i)} The integral $(I^{\a}_{+, 2} f)(t)$ is absolutely
convergent for almost all $t>0$ whenever $r\to rf(r)$ is a locally
integrable function on $\bbr_{+}$.

\textup{(ii)} If

\begin{equation} %
\label{for10z}
\intl_{a}^{\infty}|f(r)|\, r^{2\a-1}\, dr <\infty,\qquad a>0,
\end{equation}
then $(I^{\a}_{-, 2} f)(t)$ is finite for almost all $t>a$. If
$f$ is non-negative, locally integrable on $[a,\infty)$, and
(\ref{for10z}) fails, then $(I^{\a}_{-, 2} f)(t)=\infty$ for every
$t\ge a$.
\end{lemma}

Fractional derivatives of the Erd\'{e}lyi--Kober type are defined as the
left inverses ${\Cal D^{\a}_{\pm, 2} = (I^{\a}_{\pm,
2})^{-1}}$ and have different analytic expressions. For example, if
$\alpha= m + \alpha_{0}, \; m = [\alpha], \; 0 \le\alpha_{0} < 1$,
then, formally,
\begin{equation} %
\label{frr+z}
\Cal D^{\a}_{\pm, 2} \vp=(\pm D)^{m +1}\, I^{1 - \alpha_{0}}_{
\pm, 2}\vp, \qquad D=\frac{1}{2t}\,\frac{d}{dt}.
\end{equation}
More precisely,  the following statements hold.

\begin{theorem}\label{78awqe555} Let $\vp= I^{\a}_{+, 2} f$, where $rf(r)$ is locally integrable on $\bbr_{+}$.
Then $f(t)= (\Cal D^{\a}_{+, 2} \vp)(t)$ for
almost all $t\in\bbr_{+}$, as in (\ref{frr+z}).
\end{theorem}

\begin{theorem}\cite[Theorem 2.44]{Ru3}
\label{78awqe} If $f$ satisfies (\ref{for10z})
for every $a>0$ and
$\vp\!= \!I^{\a}_{-, 2} f$, then $f(t)= (\Cal D^{\a}_{-, 2} \vp)(t)$ for
almost all $t\in\bbr_{+}$, where $\Cal D^{\a}_{-, 2} \vp$ can be represented as follows.

\noindent
\textup{(i)} If $\a=m$ is an integer, then
\begin{equation} %
\label{90bedr}
\Cal D^{\a}_{-, 2} \vp=(- D)^{m} \vp,
\qquad D=\frac{1}{2t}\,\frac{d}{dt}.
\end{equation}

\noindent
\textup{(ii)} If $\alpha= m +\alpha_{0}, \; m = [ \alpha], \; 0 <
\alpha_{0} <1$, then
%
\begin{equation} %
\label{frr+z33}
\Cal D^{\a}_{-, 2} \vp= t^{2(1-\a)} (- D)^{m +1}
t^{2\a}\psi, \quad\psi=I^{1-\a+m}_{-,2} \,t^{-2m-2}\,
\vp.
\end{equation}
In particular, for $\a=k/2$, $k$ odd,
\begin{equation} %
\label{frr+z3}
\Cal D^{k/2}_{-, 2} \vp= t\,(- D)^{(k+1)/2} t^{k}I^{1/2}_{-,2}
\,t^{-k-1}\,\vp.
\end{equation}
\end{theorem}

These fractional integrals and their derivatives possess the semi-group property
\be
\bbd_n^\a \bbd_n^\b=\bbd_n^{\a+\b},\qquad I_n^\a I_n^\b=I^{\a+\b}_n,\,\nonumber
\ee
\bea\label{kauky}
\Cal D^{\a}_{\pm, 2}\Cal D^{\b}_{\pm, 2} = \Cal D^{\a+\b}_{\pm, 2}, \qquad I^{\a}_{\pm, 2}I^{\b}_{\pm, 2} = I^{\a+\b}_{\pm, 2},\eea
in suitable classes of functions that guarantee the existence of the corresponding expressions.

\subsection{Radon transforms on Grassmanians}
 Let $\agnj$ and $\gnj$ be the affine Grassmann manifold of all non-oriented $j$-planes $\t$  and the ordinary Grassmann manifold of $j$-dimensional subspaces $\xi$ of $\rn$, respectively.
Similarly, $\agnk$ and $\gnk$ denote the set of  $k$-planes $\z$ and the set of $k$-subspaces $\eta$. Then every $j$-plane $\t$ can be parameterized by the pair $(\xi,u)$, where $\xi\in \gnj$ and $u\in \xi^\perp$. Similarly, we write
$\z=(\eta,v)\in \agnk$ where $\eta\in \gnk$ and $v\in \eta^\perp$. For any integer $m$ and subspace $X$ in $\rn$, denote by $G_m(X)$ and $G(m,X)$ the sets of $m$-dimensional  subspaces and $m$-dimensional planes  in $X$, respectively. If $P$ and $Q$ are two orthogonal subspaces with no other intersection except for the original point, let $[P,Q]$  denote the
smallest subspace that contains both $P$ and $Q$.
     For incident planes  $\z=(\eta,v)\in \agnk$ and   $\t=(\xi,u)\in \agnj$, denote  $P=\xi\cap \eta$, $Q=P^\perp\cap \xi$ and $R=P^\perp\cap \eta$. Then $P$ is a $p$-dimensional subspace, $Q$ a $q$-dimensional subspace and $R$ a $l$-dimensional subspace, where $p+q=j$ and $p+l=k$. Obviously,  $\xi=[P,Q]$, $\eta=[P,R]$.
Then for $\z=(\eta,v)\in \agnk$ and $\t=(\xi,u)\in \agnj$, we have
\bea\label{inc}
\hat \z=\{[P,Q]+v+u: P\in G_p(\eta), Q\in G_q(\eta^\perp),\text{and}\,\, u\in P^\perp\cap \eta\,\},
\eea
\bea\label{inc2}
\hat \t=\{[P,R]+u+v: P\in G_p(\xi), R\in G_l(\xi^\perp),\text{and}\,\, v\in P^\perp\cap \xi\,\},
\eea
where   $p\ge 0, q\ge 0, l\ge 0$ satisfying $p+q=j$, $p+l=k$, and  $G_q(\eta^\perp)$ represents the $q$-dimensional subspace in $\eta^\perp$, $G_l(\xi^\perp)$ the $l$-dimensional subspace in $\xi^\perp$.  The manifolds $\agnj$ and $\agnk$ will be endowed with $d\t=d\xi du$ and $d\z=d\eta dv$, where $d\xi$ and $d\eta$ are the corresponding  $O(n)$-invariant measure on $\gnj$ and $\gnk$ with total mass $1$ and $du$, $dv$  the usual Lebesgue measure on subspace $\xi^\perp$ and $\eta^\perp$, respectively.
Then the orthogonal Radon transform $\rpq$ in (\ref{defor}) can be rewritten   in the following form,
\bea\label{rpq}
\rpq f(\z)=\intl_{G_q({\eta^\perp})}d_{\eta^\perp} Q\intl_{G_{p}(\eta)} d_\eta P\intl_{P^\perp\cap \eta} f([P,Q]+u+v)du\,.
\eea
where $d_{\eta^\perp} Q$ and $d_\eta P$ denote the corresponding  invariant measure on $G_q({\eta^\perp})$ and $G_{p}(\eta)$ with total mass $1$.
Similarly, we can define the corresponding dual transform. For a function $g=g(\z)$ on $\agnk$,  the dual transform of $\rpl g$ is a function on $\agnj$, defined by
\bea\label{rpl}
\rpl g(\xi,u)=\intl_{G_l({\xi^\perp})}d_{\xi^\perp} L\intl_{G_p(\xi)} d_{\xi} P\intl_{P^\perp\cap \xi} g([P,L]+u+v)dv\,.
\eea
where $d_{\xi^\perp} L$ and $d_{\xi} P$ denote the corresponding  invariant measure on $G_l({\xi^\perp})$ and $G_{p}(\xi)$ with total mass $1$, respectively.

 When $p=0$, in the meantime  $q=j$ and $l=k$,  transforms (\ref{rpq}) and (\ref{rpl}) reduce to the Gonzalez  type orthogonal Radon transform,
 \bea
 (R_j^kf)(\eta,v)=\intl_{G_j({\eta^\perp})}d_{\eta^\perp} \xi\intl_{\eta} f(\xi+u+v)du\,,
 \eea
 and its dual transform
 \bea
 (R_k^j f)(\xi,u)=\intl_{G_k({\xi^\perp})}d_{\xi^\perp} \eta\intl_{\xi} f(\eta+u+v)dv\,,
 \eea
 which studied in  \cite{G1,G2}; If moreover $j=0$,  we get the $k$-plane  transform and its dual:
 \be\label{rtra1kfty} (R_k f)(\eta,v)= \intl_{\eta} f(u+v)\, d u, \qquad  \z=(\eta,v) \in G(n,k), \quad 1\le k\le n-1;\ee
 \be  \label{mmsdcrt} (R_k^* \vp)(x)=\intl_{O(n)}\vp(\rho\eta_0 +x) \,d\rho, \ee
where $\eta_0$ is an arbitrary fixed $k$-plane through the origin and $d\rho$ the invariant measure with mass 1. The  dual
$k$-plane transform $R_k^*$  averages a function $\vp$ on $G(n,k)$  over all $k$-planes  passing through a fixed point $x\in \rn$. When $\vp=\vp(\xi,u)$ is a good radial function on  $\agnk$,  $\vp(\eta,v)=\vp_0(|v|)$, then
\bea
R_k^*\vp (x)=\frac{\sig_{k-1}\sig_{n-k-1}}{\sig_{n-1}r^{n-2}}\intl_0^r\vp_0(s)(r^2-s^2)^{k/2-1}s^{n-k-1}\,ds,\quad r=|x|.
\eea

 If $p=j$, in the meantime $q=0$, our transform reduces to the ``inclusion" Radon transform
  \bea\label{rkr}
 (R_{j,k}f)(\eta,v)=\intl_{G_j({\eta})}d_{\eta} \xi\intl_{\xi^\perp\cap \eta} f(\xi+u+v)du\,,
 \eea
 and its dual
 \bea
 (R^*_{j,k}f)(\xi,u)=\intl_{\xi\subset\eta}f(\eta+u)d_{\xi} \eta\,,
 \eea
 that integrates function $f$ on $\agnj$ over all the $k$-planes containing $j$-plane $(\xi,u)$. When $f$ is a radial function $f(\xi,u)=f_0(|u|)$, then  $R_{j,k}f(\eta,v)$ is also radial. Explicitly,
  \bea\label{rar}
  R_{j,k}f(\eta,v)=\sig_{k-j-1}\intl_s^\infty f_0(r)(r^2-s^2)^{(k-j)/2-1}r\,dr, \quad s=|v|,
  \eea
  see \cite{Ru2} for detailed informations.

\subsection{Intertwining relations and Fuglede equality}In this part, we recall  the known relations between fractional integrals and the Radon transforms on Grassmannians.
The following  theorem is due to Rubin \cite{Ru1}, that Radon-John $j$-plane transforms and their duals interwine Riesz Potentials on the source space and target space.
\begin{theorem}\cite[Section 3]{Ru1}\label{fug2}For  functions $f$ on $\rn$, $\psi$ on $\agnk$,
\be\label{ktyu}
R_k I_n^\a f\!=\!I_{n-k}^\a R_k f\,,\quad R_k^* I_{n-k}^\a \psi\!=\!I_{n}^\a R_k^* \psi\,.\ee
Here $I_{n-k}^\a$ stands  for the Riesz potential  on the $(n-k)$-dimensional fiber of the Grassmannian bundle $\agnk$. As above, it is assumed  that either side of the corresponding equality exists in the Lebesgue sense.
\end{theorem}
A similar statement for the orthogonal Radon transform  $R_j^k$ of the Gonzalez type  was proved by the author and Rubin \cite{RW2};
\begin{theorem}\cite[section4]{RW2} \label {dwalg} If $\;0<\a<n-k-j$, then
\be\label{wleiy}
(I_{n-k}^{\a} R_j^k f)(\z)= (R_j^k I_{n-j}^{\a} f)(\z), \qquad \z\in \agnk,\ee
provided that either side of this equality exists in the Lebesgue sense.
\end{theorem}
The case of $j=0$  agrees with the equality (\ref{ktyu}).
When  $j+k=n-1, m=1$, the Laplacian form of  (\ref{wleiy}) was proved by Gonzalez under the assumption that $f$ is infinitely differentiable and compactly supported; cf. \cite[Lemma 3.3] {G2}.

At last, we introduce the Fuglede's formula, which plays an important role in the construction of original function $f$ from its $k$-plane transform $R_kf$.  \begin{theorem}\label{fug}\cite{Fug}\label{howa1fu} For any   $1 \le k \le n-1$,
\be\label{kryafu} R_k^*R_kf=c_{k,n}\,I_n^{k} f, \qquad c_{k,n}=\frac{2^{k}\pi^{k/2} \Gam (n/2)}{\Gam ((n-k)/2)}, \ee
provided that either side of this equality exists in the Lebesgue sense.
\end{theorem}
Similar statements have been proved for ``Gonzalez" transform and ``inclusion" transform.

\be\label{heit1} R_k^* \,R_j^k R_j h \!=\! R_j^* \,R_k^j R_k h \!=\! c \,I_n^{j+k} h, \qquad c\!=\!\frac{2^{j+k} \pi^{(j+k)/2}\, \Gam (n/2)}{\Gam ((n-j-k)/2)};\ee
\be\label{roiw} R_k^*  R_{j,k}  R_j h=R_k^* R_k h=c_{k,n}\,I_n^{k} h, \qquad c_{k,n}=\frac{2^{k}\pi^{k/2} \Gam (n/2)}{\Gam ((n-k)/2)}, \ee
both under the conditions that the Riesz potentials on the right-hand side exist in the Lebesgue sense.
In this paper, we will  generalize them  to Strichartz type transforms, see Theorem \ref{if}.

In the following statement, we assume  the integer numbers $p,q,l$ satisfying
\be p>0,\quad q>0,\quad l>0 \quad \text{and}\quad p+q+l<n\,.\ee

\section{Radon transforms of Radial functions}

We recall that a function $f$ on $\agnj$  is  radial, if there is a function $f_0$ on $\bbr_+$, such that  $f(\t)=f_0(|\t|)$.
 If radial function $f$ is good enough, then, by  (\ref{rpq}), we can write
 \bea\label{rc}
 (\rpq f)(\eta,v)=\intl_{G_q({\eta^\perp})}F(Q+v)d_{\eta^\perp} Q\,,
\eea
 where
 \bea\label{rc1}  F(Q+v)=\intl_{G_{p}(\eta)} d_\eta P\intl_{P^\perp\cap \eta} f_0(|u+\Pr{}_{Q^{\perp}}v|)du\,,
 \eea
  and $\Pr{}_{Q^{\perp}}v$ denotes the orthogonal projection of $v$ onto $Q^\perp$. For fixed $Q$,  the expression of $F(Q+v)$  can be seen as a  $p$-plane to $k$-plane inclusion  transform  restricted in subspace $Q^\perp$. Explicitly,
   \bea
   F(Q+v)=R_{p,k}\tilde{f} (\eta,\Pr_{Q^\perp}v),
   \eea
   where $R_{p,k}$ denotes the $p$-plane to $k$-plane inclusion transform in $Q^\perp$ and function $\tilde{f}$ is a radial function on $G(p,Q^\perp)$ satisfying $\tilde{f}(P,\om)=f_0(|\om|)$.
   Using \cite[Lemma 2.3]{Ru2},
  $F(Q+v)$ is also a radial function on $G(q,\eta^\perp)$, we write  $F(Q+v)=F_0(|\Pr{}_{Q^{\perp}}v|)$. According to the formula (\ref{rar}),
  \bea\label{F0}
  F_0(s)=\sig_{l-1}\intl_s^\infty f_0(r)(r^2-s^2)^{l/2-1}r\,dr\,.
  \eea
From (\ref{rc}),  for fixed $\eta$, the function  $\rpq f(\eta,\cdot)$ is a dual $q$-plane transform  restricted in subspace $\eta^\perp$. Then from (\ref{rkr}), $\rpq f(\eta,\cdot)$ is also a radial function with the expression
  \bea\label{R0}
  \rpq f(\eta,v)=\frac{\sig_{q-1}\sig_{n-k-q-1}}{\sig_{n-k-1}t^{n-k-2}}\intl_0^t F_0(s)(t^2-s^2)^{q/2-1}s^{n-k-q-1}\,ds\,, \quad  t=|v|.
  \eea
Combining (\ref{F0}) and (\ref{R0}), using (\ref{as34b12}) and (\ref{eci}), we have the following theorem.
\begin{theorem}\label{pr}
 If $f(\t)\equiv f_0(|\t|)$ satisfies the conditions
\be\label{pf4e}
\intl_0^a |f_0 (t)|\, t^{n-j-1}dt <\infty \quad \text{and} \quad \intl_a^\infty |f_0 (t)|\, t^{l-1}dt <\infty\ee
for some $a>0$, then
\be\label{pffrr} (\rpq f)(\z)=(I_{j,k} f_0)(|\z|),\ee where
\bea
(I_{j,k} f_0)(s)\!\!&=&\!\!\frac{c_1}{s^{n-k-2}}\intl_0^s(s^2\!-\!r^2)^{q/2-1} r^{n-k-q-1}dr\!\intl_r^{\infty}\!f_0(t)(t^2\!-\!r^2)^{l/2-1} t dt\nonumber\\
\label{pffrrt} &=&\frac{\tilde c_1}{s^{n-k-2}} (I^{q/2}_{+, 2} \, r^{\ell-2} I^{l/2}_{-, 2} f_0)(s), \quad \ell=n-k-q\ge 1,\quad\eea
\be\label{pjjj4e} c_1=\frac{\sig_{l-1}\sig_{q-1}\sig_{\ell-1}}{\sig_{n-k-1}}, \qquad \tilde c_1=\frac{\pi^{l/2}\, \Gam ((n-k)/2)}{ \Gam (\ell/2)}.\ee
Moreover,
\be\label{puor} \intl_\a^\b |(I_{j,k} f_0)(s)|\,ds <\infty\quad \text{for all}\quad 0<\a<\b<\infty.\ee
\end{theorem}

\begin{proof} We just need   to show that the assumptions in (\ref{pf4e}) imply (\ref {puor}).   Then  $(I_{j,k} f_0)(s)<\infty $ for almost all $s>0$,  and therefore  (\ref{pffrr}) is meaningful.  The proof of (\ref{pf4e}) is very similar to the proof in \cite[Lemma 3,1]{RW1}. To keep the paper complete and readable, we give its proof here.
 It suffices to assume $f_0 \ge 0$.
Let $\ell =n-q-k\ge 1$ and the letter $c$ stands  for a constant that can be different  at each occurrence.
Then through (\ref{R0}), for any $0<\a<\b<\infty$,
\bea\label{ap1}
&&\intl_\a^\b (I_{j,k} f_0)(s)\,ds\le c \intl_\a^\b ds\intl_0^s(s^2\!-\!r^2)^{q/2-1} r^{\ell-1}
F_0(r)\,dr\nonumber\\
&&\le c\intl_\a^\b ds\intl_0^s(s\!-\!r)^{q/2-1} r^{\ell-1}F_0(r)\,dr\nonumber\\
&&= c\intl_0^\a  r^{\ell-1} F_0(r)\, dr\intl_\a^\b (s-\!r)^{q/2-1}ds+c\intl_\a^\b  r^{\ell-1}F_0(r)  dr\intl_r^\b (s-\!r)^{q/2-1}ds\nonumber\\
&&\le c\intl_0^\b  r^{\ell-1}  (\b-\!r)^{q/2}F_0(r)\,dr-c\intl_0^\a  r^{\ell-1}  (\a-\!r)^{q/2}F_0(r)\,dr.\nonumber
\eea
These  integrals have the same form, so we just need to show that
\bea
I(\a)\equiv \intl_0^\a  r^{\ell-1}  (\a-\!r)^{q/2}F_0(r)\,dr<\infty\,\qquad \forall \;\a>0.\nonumber
\eea
Indeed,
\bea
&&I(\a)\le \intl_0^\a  r^{\ell-1}  (\a-\!r)^{q/2}dr\intl_r^{\infty}\!f_0(t)(t-\!r)^{l/2-1} t^{l/2} dt\nonumber\\
&&\le\! c\intl_0^\a \!f_0(t)\,t^{l/2+\ell-1}dt\! \intl_0^t(t\!-\!r)^{l/2-1} dr\!+\!c\!\intl_\a^\infty \! f_0(t)\,t^{l/2} dt\!\intl_0^\a (t\!-\!r)^{l/2-1} dr\nonumber\\
&&=c \intl_0^\a f_0(t)t^{l+\ell-1}dt+c\intl_\a^\infty  f_0(t)\,t^{l/2} (t^{l/2}-(t-\a)^{l/2}) \,dt\,,\nonumber
\eea
So
\bea
I(\a)\le c\intl_0^\a f_0(t)t^{n-j-1}dt+c\intl_\a^\infty  f_0(t)\,t^{l-1} dt<\infty\,,\nonumber
\eea
The last expression is finite by  (\ref{pf4e}).
\end{proof}
\begin{remark}\label{nese}${}$
\vskip 0.2 truecm
\noindent 1. If the first integral in  (\ref{pf4e}) is not finite,  function $f(\t)=\frac{1}{|\t|^{n-j}}$  only satisfing the second integral condition in (\ref{pf4e}) gives an example that  its Strichartz transform $\rpq f(\eta,v)=\infty$ for almost every point $(\eta,v)\in \agnk$.
\vskip 0.2 truecm
\noindent 2. Through Lemma \ref{lifa2},  to guarantee the  existence of the corresponding integrals $(I_{j,k} f_0)(s)$, the finiteness of the second integral in  (\ref{pf4e})  is essentially  necessary.
\end{remark}

The following  analogue of Theorem \ref{pr} for the dual transform $\rpl g$ follows from Theorem \ref{pr} by the  symmetry.

\begin{theorem}\label{pr22}
 If $g(\z)\equiv g_0(|\z|)$ satisfies the conditions
\be\label{pf4ez}
\intl_0^a |g_0 (s)|\, s^{n-k-1}ds <\infty \quad \text{and} \quad \intl_a^\infty |g_0 (s)|\, s^{q-1}ds <\infty\,,\ee
for some $a>0$, then
\be\label{pffrr1} (\rpl g(\t)=(I_{k,j}g_0)(|\t|)\,,\ee
 where
\bea
(I_{k,j}g_0)(t)\!\!&=&\!\!\frac{c_2}{t^{n-j-2}}\intl_0^t(t^2\!-\!r^2)^{l/2-1} r^{\ell-1}dr\!\intl_r^{\infty}\!g_0(s)(s^2\!-\!r^2)^{q/2-1} s ds\nonumber\\
\label{pffrr1t} &=&\!\!\frac{\tilde c_2}{t^{n-j-2}} (I^{l/2}_{+, 2} \, r^{\ell-2} I^{q/2}_{-, 2} g_0)(t), \quad \ell=n-j-k\ge 1,\quad \eea
\be\label{pjjj4e} c_2=\frac{\sig_{l-1}\sig_{q-1}\sig_{\ell-1}}{\sig_{n-j-1}}, \qquad \tilde c_2=\frac{\pi^{q/2}\, \Gam ((n-j)/2)}{ \Gam (\ell/2)}.\ee
Moreover,
\be\label{puor1} \intl_\a^\b |(I_{k,j}g_0)(t)|\,dt <\infty\quad \text{for all}\quad 0<\a<\b<\infty.\ee
\end{theorem}

\begin{remark}
A similar statement as in Remark \ref{nese} explains   the importance of the  two integral conditions  in  (\ref{pf4ez}).
\end{remark}
\begin{example}\label{de}The following formulas can be easily obtained from (\ref{pffrrt}) and (\ref{pffrr1t}),\\
\noindent {\rm (i)} If $f(\t)\!=\! |\t|^{-\lam}$, $\!l <  \lam <\! n-j$, then $(\rpq f)(\z)\!=\!c_1 \,|\z|^{l-\lam}$;\\
\noindent {\rm (ii)} If $g(\z)\!=\! |\z|^{-\lam}$, $\!q <  \lam <\! n-k$, then $(\rpl g)(\t)\!=\!c_2 \,|\t|^{q-\lam}$;\\
\noindent {\rm (iii)} If $f(\t)\!=\! (1+|\t|^2)^{-\frac{n-p}{2}}$, then $(\rpq f)(\z)\!=\!c_3 \,(1+|\z|^2)^{-\frac{n-k-q}{2}}$;\\
\noindent {\rm (iv)} If $g(\z)\!=\! (1+|\z|^2)^{-\frac{n-p}{2}}$, then $(\rpl g)(\t)\!=\!c_4 \,(1+|\t|^2)^{-\frac{n-j-l}{2}}$,\\
where
$$
c_1=\frac{\pi^{l/2}\Gam((\lam-l)/2)\Gam((-\lam+n-j)/2)\Gam((n-k)/2))}{\Gam(\lam/2)\Gam((-\lam+n-q)/2)\Gam((n-k-q)/2)},$$
$$c_2=\frac{\pi^{q/2}\Gam((\lam-q)/2)\Gam((-\lam+n-k)/2)\Gam((n-j)/2))}{\Gam(\lam/2)\Gam((-\lam+n-l)/2)\Gam((n-j-l)/2)},$$
$$c_3=\frac{\pi^{l/2}\Gam((n-k)/2)}{\Gam((n-p)/2)},\quad c_4=\frac{\pi^{q/2}\Gam((n-j)/2)}{\Gam((n-p)/2)}.
$$
\end{example}

\begin{theorem}\label {wiuytr}  Suppose  $p+q+l<n$ and $p+q=j, p+l=k$,  and let
 $f(\t)\equiv f_0 (|\t|)$   be a  radial function on $\agnj$ satisfying
 the same integral conditions  (\ref{pf4e}). Then function $f_0$ can be recovered from the Radon transform $(\rpq f)(\z)\equiv (I_{j,k} f_0)(|\z|)$  by the formula
\be\label{sjue}
f_0(t)=\tilde c_1{\!}^{-1}\,(\Cal D^{l/2}_{-, 2} \, r^{2-\ell} \,\Cal D ^{q/2}_{+, 2} \,s^{n-k-2}\,I_{j,k} f_0)(t),\ee
where  $\tilde c_1=\pi^{l/2}\, \Gam ((n-k)/2)/\Gam (\ell/2)$ and the  Erd\'{e}lyi--Kober fractional derivatives $\Cal D^{l/2}_{-, 2}$ and $ \Cal D ^{q/2}_{+, 2}$ are defined by  (\ref{frr+z})-(\ref{frr+z3}).
\end{theorem}
\begin{proof} By (\ref{pffrrt}),
\be\label{lrtvb}  (I_{j,k} f_0)(s)\!=\!\tilde c_1 s^{2+k-n} (I^{q/2}_{+, 2} \, r^{\ell-2} I^{l/2}_{-, 2} f_0)(s).\ee
 Using the finiteness of the integrals in (\ref{pf4e}), through a  simple calculation, we can prove
 $r^{\ell-1} I^{l/2}_{-, 2} |f_0| \in L^1_{loc} (\bbr_+)$. It follows that the conditions of Theorems \ref{78awqe555} and  \ref{78awqe} are satisfied and  both fractional integrals in (\ref{lrtvb}) can be inverted. This gives  (\ref{sjue}).
\end{proof}

Interchanging $j$ and $k$, the reader can easily  arrive the following  similar statement  for the dual transform $\rpl g$.
\begin{theorem}\label {wiuytr2}  Suppose  $p+q+l<n$ and $p+q=j, p+l=k$,  and let
 $g(\z)\equiv g_0 (|\z|)$   be a   radial function on $\agnk$ satisfying the same  integral conditions   (\ref{pf4ez})\,.
  Then function $g_0$ can be recovered from the Radon transform $(\rpl g)(\t)\equiv (I_{k,j} g_0)(|\z|)$  by the formula
\be\label{sjue0}
g_0(t)=\tilde c_2{\!}^{-1}\,(\Cal D^{q/2}_{-, 2} \, r^{2-\ell} \,\Cal D ^{l/2}_{+, 2} \,s^{n-j-2}\,I_{k,j} g_0)(t),\ee
where  $\tilde c_2=\pi^{q/2}\, \Gam ((n-j)/2)/\Gam (\ell/2)$ and the  Erd\'{e}lyi--Kober fractional derivatives $\Cal D^{q/2}_{-, 2}$ and $ \Cal D ^{l/2}_{+, 2}$ are defined by  (\ref{frr+z})-(\ref{frr+z3}).
\end{theorem}

\section{Existence theorem}

\subsection{Duality}In this section, we prove the dual relationship between $\rpq$ and $\rpl$ first.   Using the idea of double   fibration ( cf. \cite{He}, P57,  or \cite{RW2}, Lemma 3.4), we can prove the following
 identity, which combined the dual relationship in \cite[Lemma 3.4]{RW2} and \cite[Lemma 2.1]{Ru2} \,.
\begin{theorem}\label{dut}Suppose  $p+q+l<n$ and $p+q=j, p+l=k$, for $f=f(\xi,u)$ on $\agnj$ and $g=g(\eta,v)$ on $\agnk$,
\bea\label{dr}
<\rpq f,g>=<f,\rpl g>,
\eea
 the equality holds if any side of the integral exists in the Lebbesgue sense.
\end{theorem}
\begin{proof}Let
\bea
\rk=\bbr e_1+\bbr e_2+\cdots + \bbr e_k,~\rnk=\bbr e_{k+1}+\bbr e_{k+2}+\cdots + \bbr e_n\,,\nn
\eea
\bea
\bbr^p=\bbr e_1+\bbr e_2+\cdots +\bbr e_p,~\bbr^q=\bbr e_{n-q+1}+\bbr e_{n-q+2}+\cdots + \bbr e_n,\nn
\eea
\bea
\bbr^l=\bbr e_{p+1}+\bbr e_{p+2}+\cdots + \bbr e_k, ~\bbr^{n-p}=\bbr e_{p+1}+\bbr e_{p+2}+\cdots +\bbr e_n.\nn
\eea
Then $\rk=[\bbr^p,\bbr^l]$, $\bbr^j=[\bbr^p,\bbr^q]$.
We begin with the following integral,
\bea\label{df}
\intl_{O(n)}d\g\intl_{\bbr^{n-p}} f_\g([\bbr^p,\bbr^q]+u)g_\g([\bbr^p,\bbr^l]+u)du,
\eea
where  $f_\gam(\t)=f(\gam \t)$, $g_\gam(\z)=g(\gam \z)$, for any $\gam\in O(n)$.
Notice that $\bbr^{n-p}=\bbr^{n-k}\oplus\bbr^l$,  the upper integral equals
\bea
\intl_{O(n)}d\g\intl_{O(k)}d\rho_1\intl_{O(n-k)}d\rho_2\intl_{\rnk}dv\intl_{\bbr^l} f_{\g\rho_1 \rho_2}([\bbr^p,\bbr^q]+ u+ v)g_{\g\rho_1 \rho_2}([\bbr^p,\bbr^l]+ v))du \nn
\eea
where  $O(k)$,  $O(n-k)$
are the stabilizers of $\rk$ , $\rnk$ in $O(n)$, and $d\rho_1, d\rho_2$ the corresponding  invariant probability  measures,  respectively. Since rotation $\rho_1\rho_2$ leaves subspaces $\rk$  and $\rnk$ invariant, so the last integral equals
\bea
&&\intl_{O(n)}d\g \intl_{\rnk}g_{\g}(\bbr^k+ v)dv\intl_{O(k)}d\rho_1\intl_{O(n-k)}d\rho_2\intl_{\bbr^l} f_{\g}([\rho_1\bbr^p,\rho_2\bbr^q]+ \rho_1 u+ v)du \nn\\
&&=\intl_{O(n)}d\g\intl_{\rnk}g_\g(\bbr^k+ v)\rpq f_\g(\bbr^k+ v)dv=\text{l.s. of ~}(\ref{dr}),\nn
\eea
where the last equality follows from the expression (\ref{rpq}).\\
Similarly, we can also prove
\bea
(\ref{df})=\text{r.s. of  ~} (\ref{dr}) \nn
\eea
Then we finish the proof of the dual equality.
\end{proof}
 Combining   Example \ref{de} and Theorem \ref{dut}, we have the following formulas.
\begin{corollary}
\bea
\intl_{\agnj}\frac{\rpl g(\t)}{|\t|^{\lam}}d\t=c_1\intl_{\agnk}\frac{g(\z)}{|\z|^{\lam-l}}d\z\,,\quad l<\lam<n-j\,;
\eea
\bea
\intl_{\agnk}\frac{\rpq f(\z)}{|\z|^{\lam}}d\z=c_2\intl_{\agnj}\frac{f(\t)}{|\t|^{\lam-q}}d\t\,,\quad q<\lam<n-k\,;
\eea
\bea
\intl_{\agnj}\frac{\rpl g(\t)}{(1+|\t|^2)^{(n-p)/2}}d\t=c_3\intl_{\agnk}\frac{g(\z)}{(1+|\z|^2)^{(n-k-q)/2}}d\z\,;
\eea
\bea\label{let}
\intl_{\agnk}\frac{\rpq f(\z)}{(1+|\z|^2)^{(n-p)/2}}d\z=c_4\intl_{\agnj}\frac{f(\t)}{(1+|\t|^2)^{(n-j-l)/2}}d\t\,,
\eea
where $c_1,c_2,c_3$ and $c_4$ are the same numbers as in Example \ref{de}.
\end{corollary}
Then we come to the following existence theorem.
\begin{theorem}${}$\hfill
\vskip 0.2 truecm

\noindent {\rm (i)} If  $f\in L^s(\agnj)$, $1\le s<(n-j)/l$, then $\rpq f(\z)$ exists for almost all $\z\in \agnk$.

\noindent {\rm (ii)} If $g \in L^r (\agnk)$, $1\le r <(n-k)/q$, then $(\rpl g)(\t)$ exists for almost all $\t\in G(n,j)$.

\noindent The bounds  $s <(n-j)/l$ and $ r <(n-k)/q$ in these statements are sharp.
\end{theorem}
\begin{proof} {\rm (i)} follows  from (\ref{let}) if we apply H\"older's inequality to the right-hand side.  If  $s\ge (n-j)/l$,
the function \[f(\t)=(2+|\t|)^{(j-n)/s}(\log(2+|\t|))^{-1}\]
which belongs to $L^s(\agnj)$ gives a counter-example, because it does not meet the second integral condition in (\ref{pf4e}). The result for $(\rpl g)(\t)$ follows if we interchange $j$ and $k$, $q$ and $l$.
\end{proof}

\section{Strichartz transforms,  Riesz Potentials and Semyanistyi  integrals}
In this part, we devote to the interesting  formulas  that connect Strichartz  Radon transforms $\rpq$, Semyanistyi type integrals and the Riesz Potentials.  Fuglede type formulas about $\rpq$ will also be considered, through which we can reconstruct function $f$ from $\rpq f$, when $f$ belongs to the range of Radon-John $j$-plane transforms.  These formulas generalize the corresponding conclusions in \cite{Ru1} and \cite{RW2}.
\subsection{Strichartz Transforms with Riesz Potentials}
\begin{theorem} \label {dwalg2} If $\;0<\a<n-k-q$, then
\be\label{wleiy2}
\rpq I_{n-j}^\a f(\eta,v)= I_{n-k}^\a\rpq f(\eta,v), \qquad (\eta,v)\in \agnk,\ee
provided that either side of this equality exists in the Lebesgue sense.
\end{theorem}
\begin{proof}Through the expression (\ref{rpq}) and (\ref{rpot}), the left hand of (\ref{wleiy2})
\bea\label{ri}
&&\rpq I_{n-j}^\a f(\eta,v)=\intl_{G_q({\eta^\perp})}d_{\eta^\perp} Q\intl_{G_p(\eta)} d_\eta P\intl_{P^\perp\cap \eta} I_{n-j}^\a f([P,Q]+u+v)du\nonumber\\
&&=\frac{1}{\g_{n-j}(\a)}\intl_{G_q({\eta^\perp})}d_{\eta^\perp} Q\intl_{G_p(\eta)} d_\eta P\intl_{P^\perp\cap \eta}du \intl_{P^\perp\cap Q^{\perp}}\frac{f([P,Q]+u+v-x)}{|x|^{n-j-\a}}dx\,.\nn
\eea
Note that $Q^\perp=\eta\oplus (Q^\perp\cap \eta^\perp)$, then the last integral equals
\bea
&&\frac{1}{\g_{n-j}(\a)}\intl_{G_q({\eta^\perp})}d_{\eta^\perp} Q\intl_{G_p(\eta)} d_\eta P\intl_{P^\perp\cap \eta}du \intl_{P^\perp\cap \eta}dx\intl_{Q^{\perp}\cap \eta^\perp}\frac{f([P,Q]+u+v-y))}{|x+y|^{n-j-\a}}dy\nn\\
&&=c_1\intl_{G_q({\eta^\perp})}d_{\eta^\perp} Q\intl_{G_p(\eta)} d_\eta P\intl_{P^\perp\cap \eta}du \intl_{Q^{\perp}\cap \eta^\perp}\frac{f([P,Q]+u+v-y)}{|y|^{n-k-q-\a}}dy\,
\eea
where $$c_1=\frac{\sig_{l-1}}{\g_{n-j}(\a)}\int_0^\infty \frac{x^{l-1}}{(1+x^2)^{(n-j-\a)/2}}dx=\frac{\Gam((n-j-l-\a)/2)}{2^{\a}\pi^{(n-j-l)/2}\Gam(\a)}.$$

Similarly, we calculate the right side of (\ref{wleiy2}),
\bea\label{ir}
&&I_{n-k}^\a\rpq f(\eta,v)=\intl_{\eta^\perp}\frac{\rpq f(\eta,v-w)}{|w|^{n-k-\a}}dw\nn\\
&&=\frac{1}{\g_{n-k}(\a)}\intl_{\eta^\perp}\frac{dw}{|w|^{n-k-\a}}\intl_{G_q({\eta^\perp})}d_{\eta^\perp} Q\intl_{G_p(\eta)} d_\eta P\intl_{P^\perp\cap \eta} f([P,Q]+u+v-w)du\nn
\eea
Note that $\eta^\perp=Q\oplus (Q^\perp\cap \eta^\perp)$, we interchange the integrals order and get
\bea
&&\frac{1}{\g_{n-k}(\a)}\intl_{G_q({\eta^\perp})}d_{\eta^\perp} Q\intl_{G_p(\eta)} d_\eta P\intl_{P^\perp\cap \eta}du \intl_{\eta^\perp\cap
Q^\perp}dy\intl_{Q}\frac{f([P,Q]+u+v-y)}{|x+y|^{n-k-\a}}dx\nn\\
&&=c_2\intl_{G_q({\eta^\perp})}d_{\eta^\perp} Q\intl_{G_p(\eta)} d_\eta P\intl_{P^\perp\cap \eta}du \intl_{\eta^\perp\cap Q^\perp}\frac{f([P,Q]+u+v-y)}{|y|^{n-k-q-\a}}dy\,
\eea
where $$c_2=\frac{\sig_{q-1}}{\g_{n-k}(\a)}\int_0^\infty \frac{x^{q-1}}{(1+x^2)^{(n-k-\a)/2}}dx=\frac{\Gam((n-k-q-\a)/2)}{2^{\a}\pi^{(n-k-q)/2}\Gam(\a)}\,.$$
Comparing the two expressions (\ref{ri}) and (\ref{ir}), we get the equality (\ref{wleiy2}).
\end{proof}
\begin{remark}  ${}$

\vskip 0.2 truecm

\noindent{1.}The differential form of (\ref{wleiy2}) was first studied by Gonzalez  under the assumption that $f$ is infinitely differentiable and compactly supported; cf. \cite[Section 4]{G4}.
\vskip 0.2 truecm

\noindent{2.}
At the case of $p=0$, that is, ``Gonzalez" case,  equality (\ref{wleiy2}) coincides  with the conclusion in   Theorem   \ref{dwalg}.
At another case $q=0$, we get the intertwining formula for the ``inclusion" Radon transforms\,.
\end{remark}
\begin{corollary} For $f=f(\xi,u)$ a function on $\agnj$,
\be\label{wleiy3}
R_{j,k} I_{n-j}^\a f(\eta,v)= I_{n-k}^\a R_{j,k} f(\eta,v), \qquad (\eta,v)\in \agnk,\ee
provided that either side of this equality exists in the Lebesgue sense.
\end{corollary}
\subsection{Strichartz transforms with Semyanistyi type integrals} In this part, we devote to the relationship between  Strichartz transforms and  Semyanistyi type integrals with the following expressions:
\bea\label{sym4}
(P_k^\a f)(\z)=\frac{1}{\gam_{n-k}(\a)}\intl_{\rn} f(x)|x-\z|^{\a+k-n}dx\,,\nonumber\\
(P_k^{\a*} \vp)(x)=\frac{1}{\gam_{n-k}(\a)}\intl_{\agnk} \vp(\z)|x-\z|^{\a+k-n}dx\,,
\eea
where $\gam_{n-k}(\a)=\frac{2^\a\pi^{(n-k)/2}\Gam(\a/2)}{\Gam((n-k-\a)/2)}$\,, $\text{Re}\, \a>0$, $\a+k-n\neq 0,2,4,...\,.$ We recall some known formulas  about the Semyanistyi integrals, which  connect them  with the $k$-plane transforms and Riesz potentials,
\bea\label{syme}
P_k^\a f=I_{n-k}^\a R_k f\,,
P_k^{\a*} \vp=R_k^* I_{n-k}^{\a}\vp\,,
\eea
where $f$ is a function on $\rn$, $g$ a function on $\agnk$ and the equalities hold if and only if either side of the formulas  exists in the Lebesgue sense\,.
Combined with  Theorem \ref{fug},
\bea\label{sym2}
P_k^{\a*}R_k f=c_{k,n}I_n^{k+\a}f\,,
\eea
\bea\label{sym3}
R_k^*P_k^{\a} f=c_{k,n}I_n^{k+\a}f\,,
\eea
where $c_{k,n}=(2\pi)^k \sig_{n-k-1}/\sig_{n-1}$\,,  see \cite[Section 3]{Ru1} for more  detailed information. Next, we will generalize these formulas to Strichartz type transform.

The following theorem connects Strichartz transforms and the dual transforms of Semyanistyi type integrals.
\begin{theorem}\label{tr0}For a function $f=f(\t)$ on $\agnj$,
\bea\label{fug0}
P_k^{\a*}\rpq f=c P_j^{(\a+l)*} f\,,\quad c=\frac{2^l \pi^{l/2}\Gam((n-j)/2)}{\Gam((n-j-l)/2)}\,,
\eea
provided that either side of this equality exists in the Lebesgue sense.
\end{theorem}
\begin{proof}
Through Theorem \ref{dwalg2}  and the second equality in (\ref{syme}), to prove Theorem \ref{tr0}, we just need to prove for a good function $f(\t)$ on $\agnj$\,, the following equality hold,
\bea\label{rrpq}
R_k^*\rpq f=c P_j^{l*} f\,,\quad c=\frac{2^l \pi^{l/2}\Gam((n-j)/2)}{\Gam((n-j-l)/2)}\,.
\eea
In fact, through the definition of  $\rpq$ (\ref{rpq}) and $R_k^*$  (\ref{mmsdcrt}),
\bea
&&R_k^*\rpq f(x)=\intl_{O(n)}\rpq f(\g \bbr^k+x)\,d\g\nn\\
&&=\intl_{O(n)}d\g\intl_{O(k)}d\rho_1\intl_{O(n-k)}d\rho_2\intl_{\bbr^l}f(\g([\rho_1 \bbr^p, \rho_2\bbr^q]+\rho_1 u)+x)du\nn\\
&&=\intl_{O(n)}d\g\intl_{\bbr^l}f(\g(\bbr^j+u)+x)du=c\intl_{\gnj}d\xi\intl_{\xi^\perp}\frac{f(\xi+u+x)}{|u|^{n-j-l}}du\,,
\nn
\eea
where $c=\frac{2^l \pi^{l/2}\Gam((n-j)/2)}{\Gam((n-j-l)/2)}$\,.
Owing to the second formula in  (\ref{syme}), equality (\ref{rrpq}) is proved. Then (\ref{fug0}) follows.
\end{proof}
\begin{remark}  ${}$

\vskip 0.2 truecm

\noindent{1.}
In the case of $p=q=0$, in the meantime $j=0$ and  $l=k$, formula (\ref{fug0}) coincides with the formula (\ref{sym2}).

\vskip 0.2 truecm

\noindent{2.}
From the proof of Theorem \ref{tr0}, we have the following intertwining relation, which connects   Strichartz  transforms with  the dual $k$-plane transforms and the dual Semyanistyi type integrals.
\end{remark}
\begin{corollary}\label{tr1}For a function $f=f(\t)$ on $\agnj$,
\bea\label{fug1}
R_k^*\rpq f=c P_j^{l*} f\,,\quad c=\frac{2^l \pi^{l/2}\Gam((n-j)/2)}{\Gam((n-j-l)/2)}\,,
\eea
provided that either side of this equality exists in the Lebesgue sense.
\end{corollary}

The following theorem
 connecting  Strichartz transforms with the  Semyanistyi type integrals.
 \begin{theorem}\label{tr21}For a suitable function $f=f(x)$ on $\rn$,
\bea
\rpq P^{\a}_j f(\eta,v)=cP_k^{q+\a}\,f(\eta,v)\,, \quad c=\frac{2^q\pi^q\Gam((n-k)/2)}{\Gam((n-k-q)/2)}\,.
\eea
provided that either side of this equality exists in the Lebesgue sense.
 \end{theorem}
 \begin{proof}
 Using Theorem (\ref{fug}), through the definition of Semyanistyi integral (\ref{sym4}),
 to prove this theorem we just need to prove the following equality holds for a good function $f$ on $\rn$,
 \bea
\rpq R_j f(\eta,v)=cP_k^q\,f(\eta,v)\,, \quad c=\frac{2^q\pi^q\Gam((n-k)/2)}{\Gam((n-k-q)/2)}\,.
\eea
In fact,
\bea
&&\rpq R_j f(\eta,v)=\intl_{G_{q}(\eta^\perp)}d_{\eta^\perp}Q\intl_{G_p{(\eta)}} d_\eta P\intl_{P^\perp\cap \eta} R_jf([P,Q]+u+v)du\nonumber\\
&&=\intl_{G_{q}(\eta^\perp)}d_{\eta^\perp}Q\intl_{G_p{(\eta)}} d_{\eta}P\intl_{P^\perp\cap \eta}du \intl_P\,dx\intl_Q f(x+y+u+v)dy\nonumber\\
&&=\intl_{G_{q}(\eta^\perp)}d_{\eta^\perp}Q\intl_{\eta}du \intl_Q f(y+u+v)dy\nonumber\\
&&=\frac{\sig_{q-1}}{\sig_{n-k-1}}\intl_{\eta}du \intl_{\eta^\perp} f(y+u+v)|y|^{k+q-n}dy\nonumber\\
&&=\frac{\sig_{q-1}}{\sig_{n-k-1}}\intl_{\rn} f(x)|x-\t|^{k+q-n}dx=cP_k^q\,f(\eta,v)\,.\nonumber
\eea
where $c=\frac{2^q\pi^q\Gam((n-k)/2)}{\Gam((n-k-q)/2)}$\,. Then we finish the proof.
 \end{proof}
 The proof of Theorem \ref{tr21} implies  the following corollary.
\begin{corollary}\label{tr2}For a suitable function $f=f(x)$ on $\rn$,
\bea\label{tre}
\rpq R_j f(\eta,v)=cP_k^q\,f(\eta,v)\,, \quad c=\frac{2^q\pi^q\Gam((n-k)/2)}{\Gam((n-k-q)/2)}\,,
\eea
provided that either side of this equality exists in the Lebesgue sense.
\end{corollary}
\begin{remark}  ${}$

\vskip 0.2 truecm

\noindent{1.}
In the case of $q=0$, Strichartz transform $\rpq$ reduces to ``inclusion" transform $\rjk$. Then  the formula (\ref{tre}) reads
\bea
\rjk R_j f(\eta,v)=R_k f(\eta,v)\,.
\eea

\vskip 0.2 truecm

\noindent{2.}
In the case of $p=0$, Strichartz transform $\rpq$ reduces to ``Gonzalez" transform $R_j^k$. Then  the formula (\ref{tre}) reads
\bea
R_j^k R_j f(\eta,v)=cP^j_k\,f(\eta,v)\,, \quad \text{where}~ c=\frac{2^j\pi^j\Gam((n-k)/2)}{\Gam((n-k-j)/2)}\,.
\eea
\end{remark}
Analogue statements still hold for the transform $\rpl$ by interchanging the index  $j$ and $k$, $q$ and $l$. In  particular, we have the following conclusions:

\noindent{\bf 1.} For a suitable function $g=g(\z)$ on $\agnk$,
\bea
P_j^{\a*}\rpl g =c P_k^{(\a+l)*}g,\quad
R_j^*\rpl g=c P_k^{l*} g\,,
\eea
where $c=\frac{2^q \pi^{q/2}\Gam((n-k)/2)}{\Gam((n-k-q)/2)}$.

\noindent{\bf 2.} For a suitable function $f=f(x)$ on $\rn$,
\bea
\rpl P^{\a}_k f(\xi,u)=cP_j^{l+\a}\,f(\xi,u)\,, \quad \rpl R_k f(\xi,u)=cP_j^{l}\,f(\xi,u)\,.
\eea
where $c=\frac{2^l\pi^l\Gam((n-j)/2)}{\Gam((n-j-l)/2)}$.
\subsection{Fuglede Type Equalities and Inversion Formulas}
\begin{theorem}
\label{if}
For a good function $f=f(x)$ on $\rn$, we have the following two formulas,
\vskip 0.2 truecm

\noindent {\rm (i)}\vspace{-1.7em}
\bea\label{rik0}
P_k^{\beta*}\rpq P^\a_j f(x)=cI_n^{\a+\beta+j+l}f(x)=cI_n^{\a+\beta+k+q}f(x)\,,
\eea

\noindent {\rm (ii)}\vspace{-1.7em}
\bea\label{rik00}
P_j^{\a*}\rpl P^\beta_k f(x)=cI_n^{\a+\beta+j+l}f(x)=cI_n^{\a+\beta+k+q}f(x)\,,
\eea
\noindent where $c=\frac{2^{j+l} \pi((j+l)/2)\Gam(n/2)}{\Gam((n-j-l)/2)}$\,,
both equalities hold provided that the Riesz potential $I_n^{\a+\beta+k+q} f$ exists in the Lebesgue sense.
\end{theorem}
From the relationship between Semyanistyi integrals and Radon-John transforms, we just need to prove the following theorem.
\begin{theorem}\label{if}
For a good function $f=f(x)$ on $\rn$,
\vskip 0.2 truecm

\noindent {\rm (i)}\vspace{-1.5em}
\bea\label{rik}
R_k^*\rpq R_j f(x)=cI_n^{j+l}f(x)=cI_n^{k+q}f(x)\,,
\eea
\vskip 0.2 truecm

\noindent {\rm (ii)}\vspace{-1.5em}
\bea\label{rik0}
R_j^*\rpl R_k f(x)=cI_n^{j+l}f(x)=cI_n^{k+q}f(x)\,,
\eea
where $ c=\frac{2^{j+l} \pi^{(j+l)/2}\Gam(n/2)}{\Gam((n-j-l)/2)}$, both equalities hold provided that the Riesz potential $I_n^{k+q} f$ exists in the Lebesgue sense.
\end{theorem}
This Theorem is obvious through  the  Lemma \ref{tr1},  Theorem \ref{fug} and Theorem \ref{fug2}.
It can also be obtained through Lemma \ref{tr2} and Theorem \ref{fug}\,.
\begin{remark}${}$\hfill

\vskip 0.2 truecm

\noindent {1. }
When $q=0$, in the meantime $j=p$, our formula (\ref{rik}) coincides with equality (\ref{roiw}) of the inclusion transform case.
\vskip 0.2 truecm

\noindent {2. }When $p=0$, in the meantime $j=q$, our formula (\ref{rik}) coincides with equality (\ref{heit1}) of the Gonzalez transform case.

\end{remark}
Theorem \ref{if} enables us to reconstruct $f$ from $\rpq f$ provided that $f$ belongs to  the range of the $j$-plane transform.

\begin{theorem} \label{qqww}  Let $f \!=\! R_j h$, $h\in L^s (\rn)$. If  $1\le s <n/(j+l)$, then
\be\label{ktyu88} f\! =\! c^{-1} R_j \bbd_n^{j+l} R_k^*\, \rpq f,  \qquad  c=\frac{2^{l+j} \pi^{(j+l)/2}\Gam(n/2)}{\Gam((n-j-l)/2)} \ee
where $\bbd_n^{j+l}$ is the Riesz fractional derivative (\ref{invs1a}). More generally, if $0<\a<n-j-l$ and  $1\le s <n/(j+l+\a)$,
 then
\be\label{ktyu88a} f\! =\! c^{-1} R_j \bbd_n^{j+l+\a} R_k^*\, I_{n-k}^\a  \rpq f\ee
with the same constant $c$.
\end{theorem}
\begin{proof} By (\ref{rik}),
\[
c^{-1} R_j \bbd_n^{j+l} R_k^* \rpq f=c^{-1} R_j \bbd_n^{j+l} R_k^* \rpq R_j h=R_j \bbd_n^{j+l} I_n^{j+l} h= R_j h=f.\]
Further, combining (\ref{rik}) with the semigroup property of Riesz potentials, we obtain
\[
c^{-1} I_n^\a R_k^* \,\rpq f \!=\! c^{-1} I_n^\a  R_k^* \,\rpq R_j h = I_n^\a  I_n^{j+l} h =   I_n^{j+l+\a} h.\]
However, by (\ref{ktyu}), $ I_n^\a  R_k^* \,\rpq f= R_k^* I_{n-k}^\a  \,\rpq f$. Hence
\[ I_n^{j+l+\a} h=  c^{-1} R_k^* I_{n-k}^\a  \,\rpq f,\]
and therefore $h= c^{-1} \bbd_n^{j+l+\a}R_k^* I_{n-k}^\a  \,\rpq f$. Applying $R_j$ to both sides, we obtain (\ref{ktyu88a}).
\end{proof}
At the case  $p=0$ and $q=0$, these inversion formulas coincide with the Theorem 5.1 and Theorem 5.2 in \cite{RW2}, respectively.
Interchanging the index  $j,k$ and $q,l$, we have the following inversion formulas for transform $\rpl$;
\begin{theorem} \label{qqww1}  Let $g \!=\! R_k h$, $h\in L^t (\rn)$. If  $1\le t <n/(k+q)$, then
\be\label{ktyu880} g\! =\! c^{-1} R_k \bbd_n^{k+q} R_j^*\, \rpl g,  \qquad  c=\frac{2^{k+q} \pi^{(k+q)/2}\Gam(n/2)}{\Gam((n-k-q)/2)} \ee
where $\bbd_n^{k+q}$ is the Riesz fractional derivative (\ref{invs1a}). More generally, if $0<\a<n-k-q$ and  $1\le t <n/(k+q+\a)$,
 then
\be\label{ktyu88a0} g\! =\! c^{-1} R_k \bbd_n^{k+q+\a} R_j^*\, I_{n-j}^\a  \rpl g\,,\ee
with the same constant $c$.
\end{theorem}

\section*{Acknowledgements} This work was originate from the valuable discussion with Professor  Boris Rubin. The author is deeply grateful to
Boris for  helpful suggestions and encouragement  on this subject.



\begin{thebibliography}{99}

\bibitem{G1}
 F. B. Gonzalez,   Radon transform
on  Grassmann manifolds, Thesis, MIT, Cambridge, MA, 1984.
\bibitem{G2}   F. B. Gonzalez,  Radon transform
on  Grassmann manifolds, J. Funct. Anal. 71 (1987) 339--362.


\bibitem{G3}   F. B. Gonzalez, Notes on Integral Geometry and Harmonic Analysis, COE Lecture Note, vol. 24
Math-for-Industry Lecture Note Series, Kyushu University, Faculty of Mathematics, Fukuoka, 2010.

\bibitem{G4} F. B. Gonzalez, Bi-invariant differential operators on the Euclidean motion group and applications to generalized Radon transforms. Ark. Mat. 26 (1988), no. 2, 191-204.

\bibitem{GK} F.B. Gonzalez, T. Kakehi, Pfaffian systems and Radon transforms on affine Grassmann manifolds,
Math. Ann.  326(2) (2003) 237-273.


\bibitem{Fug}  B. Fuglede, An integral formula,  Math. Scand. 6 (1958) 207--212.

\bibitem{He}S. Helgason,The Radon transform, Birkhauser, Boston, second edition, 1999.

\bibitem {Hel}   S. Helgason, The Radon transform on Euclidean spaces, compact two-point homogeneous spaces and Grassmann manifolds, Acta Math. 113(1965),   153-180.

\bibitem{OR}E. Ournycheva, B. Rubin, Semyanistyi's integrals and Radon transforms on matrix spaces, J. Fourier Anal. Appl. 14(1) (2008) 60-88.
\bibitem {Ru1}   B. Rubin,  Reconstruction of functions from their integrals over
 $k$-planes, Israel J. of Math. 141 (2004)  93-117.



\bibitem {Ru2} B. Rubin, Radon transforms on affine Grassmannians, Trans. Amer. Math. Soc. 356 (2004) 5045-5070.

\bibitem {Ru3} B. Rubin, Introduction to  Radon transforms: With elements of fractional calculus  and harmonic analysis, Cambridge University Press, New York, 2015.


\bibitem {RW1}B. Rubin, Y. Wang, New Inversion Formulas for Radon Transforms on Affine  Grassmannians, J. Funct. Anal. (2018) 2792-2817.


\bibitem  {RW2}B. Rubin, Y. Wang,  Riesz Potentials and  Orthogonal Radon Transforms on Affine Grassmannians, Frac. Cal. Anal. Appl. 24(2) (2021) 376-392.

\bibitem{RW3} B.Rubin and Y.Wang, Radon Transforms for Mutually Orthogonal Affine Planes,arXiv:1901.01150.

\bibitem{RW4}B. Rubin and Y. Wang, Erdelyi-Kober fractional integrals and
Radon transforms for mutually orthogonal affine planes. Fract. Calc.
Appl. Anal. 23(4) (2020), 967-979.

\bibitem{Sem}V. I. Semjanistyi,  Some integral transformations and integral geometry in an elliptic space, (Russian) Trudy Sem. Vektor. Tenzor. Anal. 12 (1963) 397-441.

\bibitem{SKM}S. G. Samko, A. A. Kilbas and O. I. Marichev, Fractional Integrals and
Derivatives. Theory and Applications, Gordon and Breach, New York, 1993.

\bibitem{Stri}R. S. Strichartz,Harmonic analysis on Grassmannian bundles, Transactions of the American Mathematical Society, 296(1) (1986) 387-409.


\end{thebibliography}
\end{document}